\documentclass[numbib]{imaiai}%%%%where imaiai is the template name
\jno{iarxxx} %%% is for doi number
\received{XX XX XXXX} %%% for received date
\revised{XX XX XXXX} %%% for revised date
\accepted{XX XX XXXX} %%% for accepted date

\usepackage{amsthm}
\usepackage{amsmath}
\usepackage{amssymb}
\usepackage{amsfonts}
\usepackage{amscd}
\usepackage{latexsym}

\usepackage{mathrsfs}
\usepackage[dvips]{graphicx}

\usepackage{color}

\usepackage{url}

\usepackage{supertabular}

\usepackage{subfigure}
\usepackage{wasysym}
\usepackage{bbm}

\DeclareMathAlphabet{\mathsfsl}{OT1}{cmss}{m}{sl}

\DeclareMathAlphabet{\mybff}{OT1}{cmr}{bx}{it}

\newcommand{\lang}{\textit}

\newcommand{\term}{\emph}

\newcommand{\cnst}[1]{\mathrm{#1}}

\renewcommand{\phi}{\varphi}

\newcommand{\eps}{\varepsilon}

\newcommand{\econst}{\mathrm{e}}

\newcommand{\zerovct}{\mathbf{0}}

\newcommand{\zeromtx}{\mathbf{0}}

\newcommand{\Id}{\mathbf{I}}

\newcommand{\R}{\mathbb{R}}

\newcommand{\abs}[1]{\left\vert {#1} \right\vert}
\newcommand{\abssq}[1]{{\abs{#1}}^2}

\newcommand{\Expect}{\operatorname{\mathbb{E}}}

\newcommand{\vct}[1]{\mybff{#1}}
\newcommand{\mtx}[1]{\mybff{#1}}

\newcommand{\adj}{*}

\newcommand{\diag}{\operatorname{diag}}
\newcommand{\trace}{\operatorname{tr}}

\newcommand{\psdle}{\preccurlyeq}

\newcommand{\norm}[1]{\left\Vert {#1} \right\Vert}
\newcommand{\normsq}[1]{\norm{#1}^2}

\newcommand{\smnorm}[2]{{\bigl\Vert {#2} \bigr\Vert}_{#1}}

\newcommand{\pnorm}[2]{\norm{#2}_{#1}}
\newcommand{\infnorm}[1]{\norm{#1}_{\infty}}

\DeclareMathSymbol{\mySigma}{\mathalpha}{letters}{"06}
\newcommand{\covmtx}{\mybff{\mySigma}}

\begin{document}

\title{The Masked Sample Covariance Estimator: \\
An Analysis via Matrix Concentration Inequalities}

\shorttitle{The Masked Sample Covariance Estimator} %%%for recto running head
\shortauthorlist{R.~Y.~Chen, A.~Gittens, and J.~A.~Tropp} %%% for verso running head

\author{{\sc Richard Y.~Chen}, \\[2pt]
{\email{ycchen@caltech.edu}} \\[6pt]
{\sc Alex Gittens} \\[2pt]
{\email{gittens@cms.caltech.edu}} \\[6pt]
{\sc and}\\[6pt]
{\sc Joel A.~Tropp} \\[2pt]
{\email{Corresponding author: jtropp@cms.caltech.edu} \\[6pt]
{Dept.~of Computing and Mathematical Sciences \\
California Institute of Technology \\
1200 E.~California Blvd., MC 305-16 \\
Pasadena, CA 91125-5000, USA}}}

\maketitle

\begin{abstract}
{Covariance estimation becomes challenging in the regime where the number $p$ of variables outstrips the number $n$ of samples available to construct the estimate.  One way to circumvent this problem is to assume that the covariance matrix is nearly sparse and to focus on estimating only the significant entries.  To analyze this approach, Levina and Vershynin (2011) introduce a formalism called \term{masked covariance estimation}, where each entry of the sample covariance estimator is reweighted to reflect an \lang{a priori} assessment of its importance.

\vspace{6pt} %\notate{return to abstract...}

This paper provides a short analysis of the masked sample covariance estimator by means of a matrix concentration inequality.  The main result applies to general distributions with at least four moments.  Specialized to the case of a Gaussian distribution, the theory offers qualitative improvements over earlier work.  For example, the new results show that $n = {\rm O}(B \log^2 p)$ samples suffice to estimate a banded covariance matrix with bandwidth $B$ up to a relative spectral-norm error, in contrast to the sample complexity $n = {\rm O}(B \log^5 p)$ obtained by Levina and Vershynin.}
{Covariance estimation, matrix concentration inequality, matrix Khintchine inequality, matrix Rosenthal inequality, random matrix, Schur product.}
%%%% If classification number provided then
\\
2010 Math Subject Classification. Primary: 60B20; Secondary: 62H12, 60F10, 60G50.
\end{abstract}

\section{Introduction}

A fundamental problem in multivariate statistics is to obtain an accurate estimate of the covariance matrix of a multivariate distribution given independent samples from the distribution.  This challenge arises whenever we need to understand the spread of the data and its marginals, for example, when we perform regression analysis~\cite{Free05} or principal component analysis~\cite{Jol02}.

In the classical setting where the number of samples exceeds the number of variables, the behavior of standard covariance estimators is well understood~\cite{JW07:Applied-Multivariate,Mui82:Aspects-Multivariate,MKB80:Multivariate-Analysis}.  The random matrix literature also contains a substantial amount of relevant work; we refer to the book~\cite{BS10:Spectral-Analysis} and the survey~\cite{Ver11} for further information.

Modern applications, in contrast, often involve a small number of samples and a large number of variables.  The paucity of data makes it impossible to obtain an accurate estimate of a general covariance matrix.  As a remedy, we must frame additional model assumptions and develop estimators that exploit this extra structure.  Over the last few years, a number of papers, including~\cite{FB07,BL08,BLe08,K08,RLZ09,CZZ10}, have focused on the situation where the covariance matrix is sparse or nearly so.  In this case, we imagine that we could limit our attention to the significant entries of the covariance matrix and thereby perform more accurate estimation with fewer samples.

This paper studies a particular technique for the sparse covariance problem that we call the \term{masked sample covariance estimator}.  This approach uses a mask matrix, constructed \lang{a priori}, to specify the importance we place on each entry of the covariance matrix.  By reweighting the sample covariance estimate using a mask, we can reduce the error that arises from imprecise estimates of covariances that are small or zero.
%For instance, if the covariance matrix decays away from the diagonal, we can use a banded mask to focus our attention on the diagonal.
The mask matrix formalism was introduced by Levina and Vershynin~\cite{LV11} to provide a unified treatment of some earlier methods for sparse covariance estimation; we refer to their paper for a more detailed discussion of prior work.

%For instance, if the covariance matrix can be approximated by a banded matrix, one sets the entries of the mask to zero outside of the band.

%Levina and Vershynin derive an elegant bound~\cite[Thm.~2.1]{LV11} for masked covariance estimation of a Gaussian distribution.  In this work, we develop a completely different analysis based on the \term{matrix Laplace transform method}~\cite{AW02,Tropp11}.  The advantage of this approach is that it applies to general subgaussian distributions and it allows us to obtain more refined information about the quality of the masked sample covariance estimator.  %In some situations, the new results provide a really significant improvement over prior work.
%
%The rest of this Introduction provides an overview of masked covariance estimation and its relationship with classical covariance estimation.  In Section~\ref{sec:intro-gauss}, we present a simplified result for the behavior of the masked sample covariance estimator applied to a Gaussian distribution, and we offer a concrete comparison with the results of Levina and Vershynin~\cite[Thm.~2.1]{LV11}.  More detailed results appear in Section~\ref{sec:result-and-proof}.

This paper provides a new analysis of the masked sample covariance estimator
using some recent ideas from random matrix
theory~\cite{Rud99:Random-Vectors,AW02,RV07:Sampling-Large,Tropp11,Tro11:Freedmans-Inequality,
Oli10a,Ver11,JZ11:Noncommutative-Bennett,MJCFT12:Matrix-Concentration}.
These methods, collectively known as \term{matrix concentration inequalities},
are particularly well suited for studying a sum of independent random matrices.
The results provide strong bounds on the moments and exponential moments of
the spectral norm of the sum by harnessing information about the individual
summands.
Indeed, matrix concentration inequalities can be viewed as far-reaching extensions
of the classical inequalities for a sum of scalar random variables~\cite{dlPG99:Decoupling}.
%due to Bernstein, Rosenthal, and other authors.
As we demonstrate in this work, matrix concentration inequalities sometimes
allow us to replace devilishly hard
calculations with simple arithmetic.

One of our main reasons for writing this paper is to show
that matrix concentration inequalities can streamline
the analysis of random matrices that arise in statistical applications.
%They can even provide good results in situations where other methods are silent.
We believe that the simplicity of our arguments and the strength of our conclusions
make a compelling case for the value of these methods.  Indeed, we hope
that matrix concentration inequalities will find a place in the toolkit of
researchers working on multivariate problems in statistics.

%\notate{The real point of this paper is to show how new ideas from RMT can be used to address problems in statistics.  It is not about statistical practice.}

\subsection{Classical Covariance Estimation} \label{sec:classical}

Consider a random vector
$$
\vct{x} = (X_1, X_2, \dots, X_p)^\adj \in \R^p.
$$
Let $\vct{x}_1, \dots \vct{x}_n$ be independent random vectors that follow the same distribution as $\vct{x}$.  For simplicity, we assume that the distribution is known to have zero mean: $\Expect \vct{x} = \zerovct$.  The \term{covariance matrix} $\covmtx$ is the $p \times p$ matrix that tabulates the second-order statistics of the distribution:
\begin{equation} \label{eqn:cov}
\covmtx := \Expect( \vct{x} \vct{x}^\adj ),
\end{equation}
where ${}^\adj$ denotes the transpose operation.  The classical estimator for the covariance matrix is the \term{sample covariance matrix}: %, which is obtained from~\eqref{eqn:cov} by the plug-in principle:
\begin{equation} \label{eqn:sample-cov}
\widehat{\covmtx}_n := \frac{1}{n} \sum\nolimits_{i=1}^n \vct{x}_i \vct{x}_i^\adj.
\end{equation}
The sample covariance matrix is an unbiased estimator of the covariance matrix:
$\Expect \widehat{\covmtx}_n = \covmtx$.

Given a tolerance $\eps \in (0, 1)$, we can study how many samples $n$ are typically required to provide an estimate with relative error $\eps$ in the spectral norm:
\begin{equation} \label{eqn:intro-rel-error-bd}
\Expect \smnorm{}{ \widehat{\covmtx}_n - \covmtx } \leq \eps \norm{ \covmtx }.
%\quad\text{with high probability}.
\end{equation}
This type of spectral-norm error bound is quite powerful.  It limits the magnitude of the estimation error for each entry of the covariance matrix; it provides information about the variance of each marginal of the distribution of $\vct{x}$; it even controls the error in estimating the eigenvalues of the covariance using the eigenvalues of the sample covariance. %because of Weyl's theorem~\cite[theorem.~III.2.1]{Bha97:Matrix-Analysis}.

Unfortunately, the error bound~\eqref{eqn:intro-rel-error-bd} for the sample covariance estimator demands a lot of samples.  Indeed, suppose that the covariance matrix has full rank.  When $n < p$, the sample covariance is rank-deficient, so the spectral norm error is bounded away from zero!
%  Then the number of samples must be at least as large as the number of variables to obtain a nontrivial guarantee.  

Typical positive results state that the sample covariance estimator is precise when the number of samples is proportional to the number of variables, provided that the distribution decays fast enough.  For example, assuming that $\vct{x}$ follows a normal distribution,
%$$
%\Expect \smnorm{}{\widehat{\covmtx}_n-\covmtx}
%	\leq \cnst{C} \left[ \sqrt{\frac{p}{n}} + \frac{p}{n} + {\rm o}(n^{-1/2}) \right] \norm{\covmtx},
%$$
%where the last term is asymptotic to zero as $n \to \infty$.  As a consequence,
\begin{equation} \label{eqn:full-samp-complexity}
n \geq \cnst{C} \, \eps^{-2} p
\quad\Longrightarrow\quad
\smnorm{}{\widehat{\covmtx}_n-\covmtx} \leq \eps \norm{ \covmtx }
\quad\text{with high probability.}
\end{equation}
We use the convention that $\cnst{C}$ denotes an absolute constant whose value may change from appearance to appearance.  See~\cite[Thm.~57 et seq.]{Ver11} for details of obtaining the bound~\eqref{eqn:full-samp-complexity}.

%The work of Srivastava and Vershynin~\cite{SV11:Covariance-Estimation} contains the most recent news on classical covariance estimation.

\subsection{The Masked Sample Covariance Estimator}
\label{sec:intro-motive}

In the regime $n \ll p$, where we have very few samples, we cannot hope to achieve an estimate like~\eqref{eqn:intro-rel-error-bd} for a general covariance matrix.  Instead, we must instate additional assumptions and incorporate this prior information to construct a regularized estimator.  Over the last few years, researchers have studied the case where the covariance matrix is sparse or nearly sparse.  In this setting, we can often refine our estimation procedure by focusing on the most significant entries of the covariance matrix.

%\vspace{6pt}
%\begin{center}
%\fbox{\begin{minipage}{0.75\textwidth}
%\emph{If we only need to estimate a small portion of the covariance matrix,
%then we can reduce the number of samples dramatically.}\end{minipage}}
%\end{center}
%\vspace{6pt}

One way to formalize this idea is to construct a
symmetric $p \times p$ matrix $\mtx{M}$ with real entries, which we call the \term{mask matrix}.
In the simplest case, the mask matrix has 0--1 values that indicate which entries of the covariance we attend to.  A unit entry $m_{ij} = 1$ means that we estimate the interaction between the $i$th and $j$th variables, while a zero entry $m_{ij} = 0$ means that we abdicate from making the estimate.  More generally, we can allow the components of the mask to range over the interval $[0, 1]$, in which case the size of $m_{ij}$ is proportional to the importance of estimating
the $(i, j)$ entry of the covariance matrix.

Given a mask $\mtx{M}$, we define the \term{masked sample covariance estimator}
$\mtx{M} \odot \widehat{\covmtx}_n$,
where the symbol $\odot$ denotes the componentwise (i.e., Schur  or Hadamard) product.  
The following expression bounds the root-mean-square spectral-norm error that this estimator incurs.
\begin{equation} \label{eqn:intro-full-err}
\left[ \Expect \smnorm{}{ \mtx{M} \odot \widehat{\covmtx}_n - \covmtx }^2 \right]^{1/2}
	\quad \leq\quad \underbrace{\left[ \Expect \smnorm{}{ \mtx{M} \odot
	\widehat{\covmtx}_n - \mtx{M} \odot \covmtx }^2 \right]^{1/2}}_{\text{variance}}
	\quad + \quad \underbrace{\smnorm{}{ \mtx{M} \odot \covmtx - \covmtx }}_{\text{bias}}.
\end{equation}
The second term in~\eqref{eqn:intro-full-err} represents the bias in the estimate owing to the presence of the mask, while the first term measures how much the estimator fluctuates about its mean value.
This bound is analogous with the classical bias--variance decomposition for the mean-squared-error (MSE)
of a point estimator.  

To obtain an effective estimator,
we must design a mask that controls both the bias and the variance in~\eqref{eqn:intro-full-err}.
We cannot neglect too many components of the covariance matrix,
or else the bias in the masked estimator may compromise its accuracy.
At the same time, each additional component we estimate contributes
to the size of the variance term.
In the case where the covariance matrix is sparse, it is natural to strike a balance
between these two effects by refusing to estimate entries of the covariance that we
know \emph{a priori} to be small or zero.

%Indeed, it is often better to make a zero estimate
%for a small component of the covariance matrix than to make an imprecise estimate.

%More generally, we can allow arbitrary nonnegative numbers in the mask.  When $m_{ij}$ is large, we study the interaction between the $i$th and $j$th variable carefully; when $m_{ij}$ is small, we are less vigilant.

Many of the regularization techniques for sparse covariance estimation studied in the literature, such as~\cite{BL08, FB07, CZZ10}, can be described using mask matrices.  These works focus on specific cases, such as banded masks and tapered masks, whereas we have followed Levina and Vershynin~\cite{LV11} by allowing an arbitrary symmetric mask $\mtx{M}$.  We refer to the papers cited in this paragraph for further background and references.

\begin{remark}[Adapted Masks]
In statistical practice, it may be more natural to \emph{estimate} the mask matrix $\mtx{M}$ from the observed samples, rather than to construct the mask \lang{a priori}.  A number of authors, including El Karoui~\cite{K08}, have studied the performance of covariance estimators with an adaptive threshold.  In this work, we focus on the simpler case where the mask is fixed.
It presents an interesting challenge to analyze a data-dependent mask using matrix concentration inequalities.
\end{remark}

\subsubsection{Example: The Banded Estimator of a Decaying Covariance Matrix}
\label{sec:banded-estimator}

Let us consider the case where the entries of the covariance matrix $\covmtx$ decay away from the diagonal.
Suppose that, for a fixed parameter $\alpha > 1$,
$$
\abs{ (\covmtx)_{ij} } \leq \abs{i - j + 1}^{-\alpha}
\quad\text{for each pair $(i, j)$ of indices}.
$$
This type of property might hold for a random process whose correlations are localized in time.  (That is, the current value of the process depends weakly on the past and the future.)  Related covariance structures arise
for random fields that have short spatial correlation scales.

A simple (suboptimal) approach to this covariance estimation problem
is to focus on a band of entries near the diagonal.
Suppose that the bandwidth $B := 2b+1$ for a nonnegative integer $b$.
For instance, a mask with bandwidth $B = 3$ for an ensemble of $p = 5$ variables takes the form
\begin{equation*} \label{eqn:band-mtx}
\mtx{M}_{\text{band}} :=
\begin{bmatrix}
	1 & 1 &  &  &  \\
	1 & 1 & 1 &  &  \\
	 & 1 & 1 & 1 &  \\
	 &  & 1 & 1 & 1 \\
	 &  &  & 1 & 1
\end{bmatrix}.
\end{equation*}
In this setting, it is straightforward to compute the bias term in~\eqref{eqn:intro-full-err}.
Indeed,
$$
\abs{ (\mtx{M} \odot \covmtx - \covmtx)_{ij} }
	\leq \begin{cases}
	\abs{i - j + 1}^{-\alpha}, & \abs{i - j} > b \\
	0, & \text{otherwise}.
\end{cases}
$$
Gershgorin's theorem~\cite[Sec.~6.1]{HJ85:Matrix-Analysis} implies that the spectral norm of a symmetric matrix is dominated by the maximum $\ell_1$ norm of a column,
so
$$
\norm{ \mtx{M} \odot \covmtx - \covmtx }
	\leq 2 \sum_{k > b} (k+1)^{-\alpha}
	\leq \frac{2}{\alpha -1} (b+1)^{1 - \alpha}.
$$
The second inequality follows when we compare the sum with an integral.
A similar calculation shows that $\norm{ \covmtx } \leq  1 + 2(\alpha - 1)^{-1}$.
Assuming the covariance matrix really does have constant spectral norm, it follows that
$$
\norm{ \mtx{M} \odot \covmtx - \covmtx }
	\lesssim B^{1 - \alpha} \norm{ \covmtx }
$$
On a relative scale, the bias decreases polynomially as we increase the bandwidth $B$ of the mask.
Note that this estimate follows from an easy application of classical matrix analysis.

We cannot complete the bound~\eqref{eqn:intro-full-err} for the estimation error
without understanding the behavior of the fluctuation term.
In contrast to the bias term,
this analysis is challenging, and it requires an excursion
into the field of random matrix theory.

\subsection{The Performance of Masked Covariance Estimation} \label{sec:intro-gauss}

This paper studies the variance term in
the error bound~\eqref{eqn:intro-full-err}
for the masked sample covariance estimator.
To perform this analysis, we must address a variety of issues:
How does the structure of the mask $\mtx{M}$ affect the performance of the estimator?
How many samples $n$ do we need to control the size of the fluctuation?
What role does the distribution of the underlying
random vector $\vct{x}$ play? % in the error term?

In Section~\ref{sec:result-and-proof}, we use a matrix
concentration inequality to obtain a bound for the
variance term in~\eqref{eqn:intro-full-err}
that holds for any distribution on $\vct{x}$ with four
finite moments.  In the Introduction, we focus on the simpler setting where
the random vector follows a normal distribution.  Our theory for this case
highlights the factors that affect the performance of the estimator.
We examine how the structure of the mask enters into the error
bound, and we describe how the error decreases with the number of samples.
We also discuss how the correlation among variables affects the difficulty
of the covariance estimate.

%We present a more general theory in Section~\ref{sec:result-and-proof}.

Note that this work focuses on the random matrix aspects of the masked
sample covariance estimator.  As a consequence,
we purposely avoid a detailed discussion of the
statistical issues.  In particular, we make no further study
of the (deterministic) bias term in~\eqref{eqn:intro-full-err}.
Nor do we presume to make any claims about statistical practice.

\subsubsection{The Complexity of a Mask} \label{sec:mask-complex}

The number of samples we need to control the variance in~\eqref{eqn:intro-full-err}
depends on ``how much'' of the covariance matrix we are attempting to estimate.
%The more interactions we are concerned with, the more samples we need to limit
%the fluctuation.
A ``complex'' mask requires us to estimate many interactions between variables, so
we need a lot of samples to limit the fluctuation of the masked
sample covariance estimator.
In this section, think about masks that take 0--1 values to gain intuition.

Our analysis identifies two separate metrics that quantify the complexity of the mask.
The matrix norm $\pnorm{1\to2}{\cdot}$ returns the maximum $\ell_2$ norm of a column.
The first complexity measure is the \emph{square} of the maximum column norm:
$$
\pnorm{1\to2}{ \mtx{M} }^2 := \max\nolimits_j \left[ \sum\nolimits_i m_{ij}^2 \right].
$$
Roughly, the bracket counts the number of interactions we want to estimate that involve the variable~$j$,
and the maximum computes a bound over all $p$ variables.  This
metric is ``local'' in nature.
The second complexity measure is the spectral norm $\norm{\mtx{M}}$ of the mask matrix,
which provides a more ``global'' view of the complexity of the interactions that we estimate.

Some examples may illuminate how these metrics reflect the properties of the mask.
First, suppose that we estimate the entire covariance matrix, so the mask is the matrix of ones:
$$
\mtx{M} = \text{matrix of ones}
\quad\Longrightarrow\quad
\pnorm{1\to2}{\mtx{M}}^2 = p
\quad\text{and}\quad
\norm{\mtx{M}} = p.
$$
%The value $p$ here corresponds with the factor $p$ in the sample complexity bound~\eqref{eqn:full-samp-complexity}.  
Next, consider the mask that arises from the banded estimator in Section~\ref{sec:banded-estimator}:
$$
\mtx{M} = \text{0--1 matrix, bandwidth $B$}
\quad\Longrightarrow\quad
\pnorm{1\to2}{\mtx{M}}^2 \leq B
\quad\text{and}\quad
\norm{\mtx{M}} \leq B
$$
because there are at most $B$ ones in each row and column.  When $B \ll p$, the banded mask
asks us to estimate fewer interactions than the full mask, so we expect the
covariance estimate to be easier.

\begin{remark}[Are the Two Metrics Really Different?]
In the examples above, the two metrics take the same value, but this coincidence does not always occur.  Although the spectral norm dominates the maximum $\ell_2$ norm of a column, the \emph{square} of the maximum column norm can be substantially larger or substantially smaller than the spectral norm.
%We have omitted examples to support this point because they do not seem to arise naturally in the setting of masked covariance estimation.
\end{remark}

\begin{remark}[Scaling]
We can reduce the size of both complexity measures by rescaling the mask,
but changing the scale of the mask may also increase the bias term in~\eqref{eqn:intro-full-err}.
When studying the masked sample covariance estimator in the context of a particular
application, it is essential to consider both the variance and the bias terms.
\end{remark}

\subsubsection{Masked Covariance Estimation for Multivariate Normal Distributions}

We are now prepared to present our main result for the case where the random
vector $\vct{x}$ follows a normal distribution with zero mean.
The statement involves the norm $\pnorm{\max}{\cdot}$,
which returns the maximum absolute entry of a matrix.

%\notate{REVIEW THE REST OF THIS SECTION AFTER FIXING REMAINDER OF PAPER!}

\begin{theorem}[Masked Covariance Estimation for a Gaussian Distribution] \label{thm:gaussian2}
Fix a $p \times p$ symmetric mask matrix $\mtx{M}$, where $p \geq 3$.
Suppose that $\vct{x}$ is a Gaussian random vector in $\R^p$ with mean zero.  Define the covariance matrix $\covmtx$ and the sample covariance matrix $\widehat{\covmtx}_n$ as in~\eqref{eqn:cov} and~\eqref{eqn:sample-cov}.  Then the variance
of the masked sample covariance estimator satisfies
\begin{multline} \label{eqn:sharp-gaussian}
\left[ \Expect \smnorm{}{\mtx{M} \odot \widehat{\covmtx}_n - \mtx{M} \odot \covmtx}^2 \right]^{1/2} \\
	\leq \cnst{C} \left[ \left( \frac{\pnorm{\max}{\covmtx}}{\norm{\covmtx}} \cdot \frac{\pnorm{1\to2}{\mtx{M}}^2 \log p}{n} \right)^{1/2}
		+ \frac{\pnorm{\max}{\covmtx}}{\norm{\covmtx}} \cdot \frac{\norm{\mtx{M}} \log p \cdot \log(np)}{n} \right] \norm{\covmtx}.
\end{multline}
%In particular,
%\begin{equation} \label{eqn:gaussian}
%\Expect \smnorm{}{\mtx{M} \odot \widehat{\covmtx}_n - \mtx{M} \odot \covmtx}
%	\leq \cnst{C} \left[ \left( \frac{\pnorm{1\to2}{\mtx{M}}^2 \log p}{n} \right)^{1/2}
%		+ \frac{\norm{\mtx{M}} \log^2 p}{n} \right] \norm{ \covmtx }.
%\end{equation}
\end{theorem}

The proof of Theorem~\ref{thm:gaussian2} appears in Section~\ref{subsec:gaussian}.
The rest of this section consists of some discussion of the result, as well as
comparisons with related work.

%Note that the transition from~\eqref{eqn:sharp-gaussian} to~\eqref{eqn:gaussian}
%follows from the inequality $\pnorm{\max}{\covmtx} \leq \norm{\covmtx}$.

\subsubsection{Extension to Distributions with Nonzero Mean}

In the actual practice of covariance estimation, we would center each sample empirically by subtracting the sample mean $\bar{\vct{x}} = n^{-1} \sum\nolimits_{i=1}^n \vct{x}_i$.  The sample covariance~\eqref{eqn:sample-cov} is computed using the centered samples $\widetilde{\vct{x}}_i = \vct{x}_i - \bar{\vct{x}}$ instead of the original samples $\vct{x}_i$.
Theorem~\ref{thm:gaussian2} can be extended to cover the masked covariance estimator formed with centered samples.  See~\cite[Rem.~4]{LV11} for the details of the argument.

\subsubsection{Discussion and Interpretation}

Theorem~\ref{thm:gaussian2} exposes several phenomena in the behavior
of the masked sample covariance estimator.  The first term
on the right-hand side of~\eqref{eqn:sharp-gaussian} reflects the
scale for moderate deviations of the estimator, and it depends on the
``local'' complexity of the mask.
%The factor $n^{-1/2}$ is analogous with the scaling in the central limit theorem.
The second term on the right-hand side of~\eqref{eqn:sharp-gaussian}
reflects the scale for large deviations of the estimator.  It
depends on the ``global'' complexity of the mask.
When the sample size $n$ is large, the first term drives
the bound because the second term usually decays faster.

The ratio of the maximum entry of the covariance matrix to the spectral norm
is an interesting feature of~\eqref{eqn:sharp-gaussian}.  The ratio never
exceeds one, but it can be as small as $p^{-1}$ when the covariance matrix
has rank one.  We interpret this factor as saying that covariance estimation
is easier when the variables are highly correlated.

\subsubsection{Sample Complexity Bound}

Markov's inequality can be used to convert~\eqref{eqn:sharp-gaussian} into an error bound that holds in probability, so
Theorem~\ref{thm:gaussian2} also allows us to develop conditions on the number $n$ of samples that we need to control the size of the fluctuation.  
%For example, with probability at least 99\%,
%\begin{equation} \label{eqn:highprob}
%\smnorm{}{\mtx{M} \odot \widehat{\covmtx}_n - \mtx{M} \odot \covmtx}
%	\leq \cnst{C} \left[ \left( \frac{\pnorm{1\to2}{\mtx{M}}^2 \log p}{n} \right)^{1/2}
%		+ \frac{\norm{\mtx{M}} \log^2(np)}{n} \right] \norm{ \covmtx }.
%\end{equation}
To obtain the sample complexity, assume that $n \leq p$,
and select an error tolerance $\eps \in (0, 1)$.
Then there is a constant $\cnst{C}_{99\%}$ for which
%~\eqref{eqn:highprob} yields the statement
\begin{equation} \label{eqn:masked-size}
n \geq \cnst{C}_{99\%}
\left[ \frac{\pnorm{1\to2}{\mtx{M}}^2 \log p}{\eps^2} + \frac{\norm{\mtx{M}} \log^2 p}{\eps} \right]
\frac{\pnorm{\max}{\covmtx}}{\norm{\covmtx}}
\quad\Longrightarrow\quad
\smnorm{}{\mtx{M} \odot \widehat{\covmtx}_n - \mtx{M} \odot \covmtx} \leq \eps \norm{ \covmtx }
\end{equation}
with probability at least 99\%.  
See the discussion after Theorem~\ref{thm:main}
for information about how to obtain sample
complexity bounds that hold with higher probability.

\subsubsection{Example: The Banded Covariance Estimator}

Consider the banded covariance estimation problem in Section~\ref{sec:banded-estimator}, with the mask
$$
\mtx{M} = \text{0--1 matrix with bandwidth $B$}.
$$
The sample complexity bound~\eqref{eqn:masked-size} and the norm calculations from Section~\ref{sec:mask-complex}
demonstrate that
\begin{equation} \label{eqn:banded-samp-comp}
n \geq \cnst{C} \left[ \frac{B \log p}{\eps^2} + \frac{B \log^2 p}{\eps} \right] \cdot \frac{\pnorm{\max}{\covmtx}}{\norm{\covmtx}}
%\quad\Longrightarrow\quad
%\smnorm{}{\mtx{M} \odot \widehat{\covmtx}_n - \mtx{M} \odot \covmtx} \leq \eps \norm{ \covmtx }
\end{equation}
is sufficient to obtain a relative spectral-norm error $\eps$ with constant probability.
In particular, the condition $n \gtrsim B \log^2 p$ always ensures a constant
relative error in~\eqref{eqn:sharp-gaussian}.
It follows that, when $B \ll p$, the variance of the estimator can be small,
even when the number of samples is much smaller than the total number of variables.

%For comparison, recall the sufficient condition~\eqref{eqn:full-samp-complexity} that the sample complexity for estimating the entire covariance matrix with relative error $\eps$ satisfies
%$$
%n \geq \cnst{C} \, \eps^{-2} p.
%$$
%When the bandwidth is much smaller than the number of variables ($B \ll p$), the masked covariance estimator outperforms the classical covariance estimator.  On the other hand, when the bandwidth is comparable with the number of variables, the analysis of the masked covariance estimator gives a sample complexity bound~\eqref{eqn:banded-samp-comp} that is worse by a polylogarithmic factor.  This effect is mitigated when the ratio of the norms in~\eqref{eqn:banded-samp-comp} is small.

%Observe that, when $\eps$ is constant, the second summand in~\eqref{eqn:banded-samp-comp} dominates the first as $p \to \infty$.  On the other hand, the first summand is larger when $\eps \leq \log^{-1} p$.  In other words, the excess logarithm in the second term of~\eqref{eqn:banded-samp-comp} does not have an impact on the sample complexity when we are seeking highly accurate covariance estimates.

\subsubsection{Comparison with Bounds of Levina and Vershynin}

The most natural point of comparison for Theorem~\ref{thm:gaussian2} is the
main theorem of Levina and Vershynin~\cite[Thm.~2.1]{LV11}.
Their result states that, for a centered normal random vector $\vct{x}$,
the fluctuation in the masked sample covariance estimator satisfies
$$
\Expect \smnorm{}{\mtx{M} \odot \widehat{\covmtx}_n - \mtx{M} \odot \covmtx}
	\leq \cnst{C} \left[ \left( \frac{\pnorm{1\to2}{\mtx{M}}^2 \log^{5} p}{n} \right)^{1/2}
		+ \frac{\norm{\mtx{M}} \log^3 p}{n} \right] \norm{ \covmtx }.
$$
The associated sample complexity bound is
\begin{equation} \label{eqn:lv-samp-comp}
n \geq \cnst{C} \left[ \frac{\pnorm{1\to2}{\mtx{M}}^2 \log^5 p}{\eps^2} + \frac{\norm{\mtx{M}} \log^3 p}{\eps}  \right].
\end{equation}
Our sample complexity bound~\eqref{eqn:masked-size} has a structure similar to~\eqref{eqn:lv-samp-comp},
but several improvements are worth mentioning.
First, the ratio of norms is a new feature in our estimate~\eqref{eqn:masked-size}.
%This represents a new phenomenon that previous authors have not quantified.
The second improvement over~\eqref{eqn:lv-samp-comp}, which has less
conceptual significance, is the reduction of the number of logarithmic factors.
Finally, our main result, Theorem~\ref{thm:main} covers all centered
distributions with four finite moments.

\subsection{Organization of the paper}

The rest of the paper is organized as follows.  Section~\ref{sec:preliminaries} introduces our notation and some preliminaries.  Section~\ref{sec:result-and-proof} presents the main result for zero-mean distributions
with finite fourth moments,
together with its proof and the proof of Theorem~\ref{thm:gaussian2}.  In Section~\ref{sec:moments}, we deal with the
technical estimates at the heart of the main result.  Appendix~\ref{app:matrix-concentration}
establishes the matrix concentration inequality we require.

\section{Preliminaries} \label{sec:preliminaries}

This section sets out the background material we need for the proof.
Section~\ref{subsec:notation} summarizes our notational conventions, and Section~\ref{subsec:hadamard} describes some basic properties of the Schur product.

\subsection{Notation and Conventions} \label{subsec:notation}

In this paper, we work exclusively with real numbers.  Plain italic letters always refer to scalars.  Bold italic lowercase letters, such as $\vct{a}$, refer to column vectors.  Bold italic uppercase letters, such as $\mtx{A}$, denote matrices.  All matrices in this work are square; the dimensions are determined by context.  We write $\zeromtx$ for the zero matrix and $\Id$ for the identity matrix.  The matrix unit $\mathbf{E}_{ij}$ has a unit entry in the $(i, j)$ position and zeros elsewhere.  

The symbol ${}^\adj$ denotes the transpose operation on vectors and matrices.  We use the term \term{self-adjoint} to refer to a matrix that satisfies $\mtx{A} = \mtx{A}^\adj$ to avoid confusion with symmetric random variables.  Curly inequalities refer to the positive-semidefinite partial ordering on self-adjoint matrices: $\mtx{A} \psdle \mtx{B}$ if and only if $\mtx{B} - \mtx{A}$ is positive semidefinite.

The function $\diag(\cdot)$ maps a vector $\vct{a}$ to a matrix whose diagonal entries correspond with the entries of $\vct{a}$.  When applied to a matrix, $\diag(\cdot)$ zeroes out the off-diagonal entries.
We write $\trace(\cdot)$ for the trace.  The symbol $\odot$ denotes the componentwise (i.e., Schur or Hadamard) product of two matrices.  

%In this paper, matrices are denoted by bold upper-case letters such as $\mtx{A}$, while (column) vectors are denoted by bold lower-case letters. The zero matrix is $\zeromtx$ and $\Id$ represents the identity matrix. The matrix $\mathbf{E}_{ij}$ has a one at the $(i,j)$ entry and zeros elsewhere. 

We write $\norm{\cdot}$ for both the $\ell_2$ vector norm and the associated operator norm, which is usually called the \term{spectral norm}.  The symbol $\pnorm{q}{\cdot}$ refers to the Schatten $q$-norm of a matrix:
$$
\pnorm{q}{\mtx{A}} := \left[ \trace \abs{\mtx{A}}^q \right]^{1/q}
%	= \left[ \trace (\mtx{A}^\adj \mtx{A})^{q/2} \right]^{1/q}.
$$
where $\abs{\mtx{A}} := (\mtx{A}^\adj \mtx{A})^{1/2}$.
The norm $\infnorm{\cdot}$ returns the maximum absolute entry of a vector, but
we use a separate notation $\pnorm{\max}{\cdot}$ for the maximum absolute entry
of a matrix.  We also require the norm
$$
\pnorm{1\to2}{\mtx{A}} := \max\nolimits_j \left( \sum\nolimits_i \abssq{ a_{ij} } \right)^{1/2}.
$$
The notation reflects the fact that this is the natural norm for linear maps from $\ell_1$ into $\ell_2$.

We reserve the symbol $\xi$ for a \term{Rademacher random variable}, which takes the two values $\pm 1$ with equal probability.  We also assume that all random variables are sufficiently regular that we are justified in computing expectations, interchanging limits, and so forth.

%A Rademacher random variable $\eps$ takes the value $1$ with probability $1/2$ and $-1$ otherwise. 

\subsection{Facts about the Schur Product} \label{subsec:hadamard}

The proof depends on some basic properties of Schur products.  The first result is a simple but useful algebraic identity.  For each square matrix $\mtx{A}$ and each conforming vector $\vct{x}$,
% and $\vct{y}$ with appropriate dimensions,
\begin{equation} \label{eqn:Hdm}
\mtx{A} \odot \vct{xx}^\adj = \diag(\vct{x}) \mtx{A} \diag(\vct{x}).
\end{equation}
The second result states that the Schur product with a positive-semidefinite matrix is order preserving.  That is, for a fixed positive-semidefinite matrix $\mtx{A}$,
\begin{equation} \label{eqn:hdm-monotone}
\mtx{B}_1 \psdle \mtx{B}_2 \quad \text{implies} \quad \mtx{A} \odot \mtx{B}_1 \psdle \mtx{A} \odot \mtx{B}_2.
\end{equation}
This property follows from Schur's theorem~\cite[Thm.~7.5.3]{HJ94}, which states that the Schur product of two positive-semidefinite matrices remains positive semidefinite.

\section{Masked Covariance Estimation} \label{sec:result-and-proof}

In this section, we state and prove detailed error estimates for masked covariance estimation
of a general distribution with finite fourth moments.
%Section~\ref{subsec:model} contains a short review of our model assumptions, and 
Section~\ref{subsec:moment-control} defines two concentration parameters that measure the spread of the distribution.  We present the main theorem and a short discussion in Sections~\ref{subsec:theorem} and~\ref{sec:discuss}.  In Section~\ref{subsec:gaussian}, we show how to derive Theorem~\ref{thm:gaussian2}, the result for Gaussian distributions.  The proof of the main result appears in Section~\ref{subsec:proof}.

\subsection{Concentration Parameters} \label{subsec:moment-control}

The effectiveness of the masked sample covariance estimator depends on the concentration properties of the distribution of $\vct{x}$.  Let us introduce two quantities that measure different facets of the variation of the random vector.

For $r \geq 1$, the \term{$r$th diagonal moment} $\mu_r(\vct{x})$ of the distribution is defined to be the maximum
$L_r$ norm of a single component of the vector:
\begin{equation} \label{eqn:muq}
\mu_r(\vct{x}) := \max\nolimits_i \left( \Expect \abs{ X_i }^r \right)^{1/r}.
\end{equation}
In other words, $\mu_r$ gives us uniform control on the $r$th moment of each component of $\vct{x}$.

We also require some information about the spread of the distribution in all directions.  Define the \term{uniform fourth moment} $\nu(\vct{x})$ by the formula
\begin{equation} \label{eqn:nu}
\nu( \vct{x} ) := \sup_{\norm{\vct{u}} = 1} \left( \Expect \abs{ \vct{u}^\adj \vct{x} }^4 \right)^{1/4}.
\end{equation}
The uniform fourth moment measures how much the worst marginal varies.

Note that both $\mu_r(\vct{x})$ and $\nu(\vct{x})$ have the same homogeneity as the random vector $\vct{x}$.  (This property is sometimes expressed by saying that the quantities have the same dimension, the same units, or the same scaling.)  As a consequence, the quantities $\mu_r(\vct{x}) \nu(\vct{x})$ and $\mu_r^2(\vct{x})$ have the same homogeneity as the covariance matrix $\covmtx$.
In the sequel, we abbreviate $\mu_r := \mu_r(\vct{x})$ and $\nu := \nu(\vct{x})$ whenever the distribution of the random vector $\vct{x}$ is clear.  

%\begin{remark}
%For Gaussian distributions, the uniform fourth moment $\nu$ always dominates the subgaussian coefficient $\kappa$.  In the worst case, $\nu$ can be much larger than $\kappa$.  Indeed, suppose that $X$ is a standard normal random variable, and consider the random vector $\vct{x} = (X, X, \dots, X)^* \in \R^p$.  Although the subgaussian coefficient $\kappa(\vct{x}) = \sqrt{2}$, the directional fourth moment $\nu(\vct{x}) = 12^{1/4} \sqrt{p}$.
%
%For other kinds of distributions, the subgaussian coefficient $\kappa$ may be substantially larger than the uniform fourth moment $\nu$.  Examples of this phenomenon already emerge in the univariate case.
%\end{remark}

\subsection{Main Result for Masked Covariance Estimation} \label{subsec:theorem}

%\notate{STILL NEED TO REWRITE 3.3}

The following theorem provides detailed information about the variance of the error in the masked sample covariance estimator for a zero-mean distribution with finite fourth moments.

%\notate{Need to rewrite 3.3 and 3.4}
%
%The following theorem, the main result of this paper, constrains the expected estimation error and characterizes the tail behavior of the error for zero-mean subgaussian distributions.

\begin{theorem}[The Masked Sample Covariance Estimator] \label{thm:main}
Fix a $p \times p$ symmetric mask matrix $\mtx{M}$, where $p \geq 3$.
Suppose that $\vct{x}$ is a random vector in $\R^p$ with mean zero.  Define the covariance matrix $\covmtx$ and the sample covariance matrix $\widehat{\covmtx}_n$ as in~\eqref{eqn:cov} and~\eqref{eqn:sample-cov}.  
Then the variance in the masked sample covariance estimator satisfies
\begin{multline} \label{eqn:theorem-main2}
\left[ \Expect \smnorm{}{\mtx{M} \odot \widehat{\covmtx}_n - \mtx{M} \odot \covmtx}^2 \right]^{1/2} \\
	\leq \sqrt{\frac{8\econst \log p}{n}} \cdot \pnorm{1\to 2}{ \mtx{M} } \cdot \mu_4 \nu 
	+ \frac{8 \econst \log p}{n} \cdot \norm{\mtx{M}} \cdot \left[ \Expect \max\nolimits_i \infnorm{\vct{x}_i}^4 \right]^{1/2}.
\end{multline}
Furthermore, the expected maximum satisfies the bound
\begin{equation} \label{eqn:thm-expect-max}
\left[ \Expect \max\nolimits_i \infnorm{\vct{x}_i}^4 \right]^{1/2}
	\leq \inf_{r \geq 1} \ (np)^{1/2r} \cdot \mu_{4r}^2.
\end{equation}
The diagonal moment $\mu_r$ and the uniform fourth moment $\nu$ are defined in~\eqref{eqn:muq} and \eqref{eqn:nu}.
\end{theorem}

In Section~\ref{sec:discuss}, we offer a short discussion of this result.  Afterward, in Section~\ref{subsec:gaussian}, we specialize the result to Gaussian distributions, which establishes Theorem~\ref{thm:gaussian2} of the Introduction.
The proof of Theorem~\ref{thm:main} appears below in Section~\ref{subsec:proof}.

\subsection{Discussion} \label{sec:discuss}

Theorem~\ref{thm:main} has a wider scope that most of the results in the literature on sparse covariance estimation.
Indeed, we allow completely general masks, and the bound is valid for any distribution with finite fourth moments.
When we specialize the result to the Gaussian case, we obtain an improvement over prior work~\cite[Thm.~2.1]{LV11}.
Even so, our argument, which is based on a matrix moment inequality, is very direct.

%Theorem~\ref{thm:main} exposes several phenomena in the behavior of the masked sample covariance estimator.  The first term on the right-hand side of~\eqref{eqn:theorem-main2} describes the scale on which moderate fluctuations occur.  It depends on the ``local'' complexity measure of the mask $\mtx{M}$, as well as the fourth moments of the distribution of $\vct{x}$.  The factor $n^{-1/2}$ is analogous with the scaling in the central limit theorem.  The second term on the right-hand side of~\eqref{eqn:theorem-main2} describes the scale for large fluctuations.  It depends on the ``global'' complexity measure of the mask, as well as the extreme value attained by the samples.
%When the sample size $n$ is large, we expect the first term to drive the bound because the second term usually decays faster.

For simplicity, we have presented Theorem~\ref{thm:main} as a bound on the variance of the masked sample covariance estimator.  A refinement of the same argument allows us to compute higher moments of the error, which in turn yield polynomial tail bounds via Markov's inequality.  When the distribution of the random vector $\vct{x}$ is subgaussian, this method even yields exponential tail bounds.

%We can also extend Theorem~\ref{thm:main} to the case where we center the observations with the sample mean before computing the sample covariance; the argument is similar to the one described by Levina and Vershynin~\cite[Rem.~4]{LV11} for the Gaussian case.

\begin{remark}[Alternative Arguments]
The proof of Theorem~\ref{thm:main} is based on a new matrix moment
inequality.  We can obtain similar results using other matrix concentration inequalities
that appear in the literature.
In particular, the matrix Rosenthal inequality~\cite[Cor.~7.5]{MJCFT12:Matrix-Concentration}
leads to a very similar bound.  The initial version of this
manuscript~\cite{CGT12:Masked-Sample-TR} uses the matrix
Bernstein inequality~\cite{Tropp11} to develop a version of Theorem~\ref{thm:main}
for a subgaussian random vector $\vct{x}$.
\end{remark}

\subsection{Specialization to Gaussian Distributions} \label{subsec:gaussian}

It is natural to apply Theorem~\ref{thm:main} to study the performance
of masked covariance estimation for a zero-mean Gaussian random vector.
In this case, the covariance matrix determines the distribution completely,
so we can obtain a more transparent statement that does not involve
the concentration parameters. %$\mu_r$ and $\nu$.
Theorem~\ref{thm:gaussian2} follows from these
considerations.

\vspace{6pt}

\begin{proof}[Proof of Theorem~\ref{thm:gaussian2} from Theorem~\ref{thm:main}]
First, we compute the $(2r)$th diagonal moment $\mu_{2r}(\vct{x})$ for $r \geq 1$.  Observe that the $i$th component $X_i$ of the vector $\vct{x}$ is a centered normal random variable with variance $\sigma_{ii}$, where $\sigma_{ii}$ denotes the $i$th diagonal entry of $\covmtx$.  Using the standard expression for the $(2r)$th moment of a normal random variable, we obtain
$$
\Expect \abs{X_i}^{2r} = \frac{(2r)!}{2^r \, r!} \cdot \sigma_{ii}^{r}
	\leq r^r \cdot \sigma_{ii}^{r}.
$$
This bound is valid for each real number $r \geq 1$.  Therefore, taking the $r$th root, we reach
\begin{equation} \label{eqn:gauss-2r}
\mu_{2r}^2 \leq r \cdot \max\nolimits_i \sigma_{ii}
	= r \pnorm{\max}{\covmtx}.
\end{equation}
The identity holds because the maximum entry of a positive-definite matrix occurs on its diagonal.
In particular, we see that
\begin{equation} \label{eqn:gauss-4}
\mu_4 \leq \sqrt{2} \pnorm{\max}{\covmtx}^{1/2}.
\end{equation}

Next, we instantiate the bound~\eqref{eqn:thm-expect-max} on the expected maximum.
Choose $4r = 2\log(np)$ to reach
\begin{equation} \label{eqn:gauss-max}
\left[ \Expect \max\nolimits_i \infnorm{\vct{x}_i}^4 \right]^{1/2}
	\leq \econst \, \mu_{2 \log(np)}^2
	\leq \econst \, \log(np) \cdot \pnorm{\max}{\covmtx},
%	\leq 2 \econst \, \log p \cdot \pnorm{\max}{\covmtx},
\end{equation}
owing to~\eqref{eqn:gauss-2r}.

Finally, we bound the uniform fourth moment $\nu(\vct{x})$.  Fix a unit vector $\vct{u}$.  The distribution of the marginal $\vct{u}^\adj \vct{x}$ is Gaussian with mean zero.  To compute the variance $\sigma_{\vct{u}}^2$ of the marginal, we write $\vct{x} = \covmtx^{1/2} \vct{g}$, where $\vct{g}$ is a standard Gaussian vector.  Then
$$
\sigma_{\vct{u}}^2
	= \Expect \abssq{ \vct{u}^\adj \vct{x} }
	= \Expect \vert \vct{u}^\adj (\covmtx^{1/2} \vct{g}) \vert^2
	= \vct{u}^\adj \covmtx^{1/2} (\Expect \vct{gg}^\adj) \covmtx^{1/2} \vct{u}
	= \vct{u}^\adj \covmtx \vct{u} \leq \norm{ \covmtx }.
$$
The fourth moment of a Gaussian variable equals three times its squared variance, so
$$
\Expect \abs{ \vct{u}^\adj \vct{x} }^4 = 3 \sigma_{\vct{u}}^4 \leq 3 \normsq{ \covmtx }.
$$
We conclude that the uniform fourth moment satisfies
\begin{equation} \label{eqn:gauss-nu}
\nu(\vct{x}) = \sup_{\norm{\vct{u}} = 1} ( \Expect \abs{ \vct{u}^\adj \vct{x} }^4 )^{1/4}
	\leq 3^{1/4} \norm{ \covmtx }^{1/2}.
\end{equation}

To complete the proof of Theorem~\ref{thm:gaussian2},
substitute the bounds~\eqref{eqn:gauss-4},~\eqref{eqn:gauss-max}, and~\eqref{eqn:gauss-nu} into
the inequality~\eqref{eqn:theorem-main2} from Theorem~\ref{thm:main}.
%The final result follows upon simplification.
\end{proof}

\subsection{Proof of Theorem \ref{thm:main}} \label{subsec:proof}

The proof of Theorem~\ref{thm:main} proceeds in several short steps.  First, we 
write the variance of the estimator as a sum of independent random matrices,
and we use symmetrization to simplify the expression.
Second, we apply a matrix moment inequality to bound the variance of the estimator
in terms of the spectral norm of a matrix variance and
the maximum spectral norm of the summands.
Finally, some short computations, which appear in
Section~\ref{sec:moments}, yield bounds for the
remaining terms.

\subsubsection{Symmetrization}

We begin by assigning a name to the quantity of interest:
$$
E := \Expect \smnorm{}{ \mtx{M} \odot \widehat{\covmtx}_n - \mtx{M} \odot \covmtx }.
$$
The random matrix inside the norm has a natural expression as a sum of independent, centered
random matrices.  To see why, substitute the definitions~\eqref{eqn:cov}
and~\eqref{eqn:sample-cov} of the population covariance matrix $\covmtx$ and the
sample covariance matrix $\widehat{\covmtx}_n$ to obtain
$$
E = \frac{1}{n} \cdot \Expect \norm{ \sum\nolimits_{i=1}^n
	(\mtx{M} \odot \vct{x}_i \vct{x}_i^\adj - \Expect \mtx{M} \odot \vct{x}_i \vct{x}_i^\adj ) }.
$$
The standard symmetrization method~\cite[Lem.~6.3]{LT91:Probability-Banach} yields the bound
\begin{equation} \label{eqn:E-symm}
E \leq \frac{2}{n} \cdot \Expect \norm{ \sum\nolimits_{i=1}^n
	\xi_i (\mtx{M} \odot \vct{x}_i \vct{x}_i^\adj) }.
\end{equation}
Here, $\{ \xi_i \}$ is a sequence of independent Rademacher random variables that is also independent
from the sequence $\{ \vct{x}_i \}$ of samples.  The advantage of the expression~\eqref{eqn:E-symm}
is that each Schur product involves a rank-one matrix, which greatly simplifies our computations.

\subsubsection{The Spectral Norm of an Independent Sum}

The main technical tool in this paper is a bound for the second moment of
the spectral norm of a sum of independent, symmetric random matrices.

%The values of the constants can be extracted with a little work.

\begin{theorem}[Matrix Second Moment Inequality] \label{thm:new-ros}
Assume that $p \geq 3$.
Consider a finite sequence $\{ \mtx{Y}_i \}$ of independent, symmetric, random, self-adjoint matrices
with dimension $p \times p$.  Then
$$
\left[ \Expect \normsq{ \sum\nolimits_i \mtx{Y}_i } \right]^{1/2}
	\leq \sqrt{2 \econst \log p} \cdot \norm{ \left[ \sum\nolimits_i \Expect \mtx{Y}_i^2 \right]^{1/2} }
	+  4\econst \log p \cdot \left[ \Expect \max\nolimits_i \norm{ \mtx{Y}_i }^2 \right]^{1/2}.
$$ 
\end{theorem}

Theorem~\ref{thm:new-ros} reduces the challenging problem of bounding the
second moment of the spectral norm to two simpler calculations.  We interpret
the first term as the variance of a sum of independent, symmetric random matrices.
The second term measures the typical size of the largest summand.
The result is new in the form that we present it,
but it has strong precedents in the literature.  See Appendix~\ref{app:matrix-concentration}
for the proof and a discussion of related work.

With Theorem~\ref{thm:new-ros} at hand,
it is straightforward to bound~\eqref{eqn:E-symm}.  We reach
\begin{align} \label{eqn:E-ros}
E &\leq \frac{1}{n} \sqrt{8\econst \log p} \norm{ \left[
	\sum\nolimits_i \Expect (\mtx{M} \odot \vct{x}_i \vct{x}_i^\adj)^2 \right]^{1/2} }
	+ \frac{8\econst \, \log p}{n} \left[ \Expect \max\nolimits_i
	\norm{ \mtx{M} \odot \vct{x}_i \vct{x}_i^\adj }^2 \right]^{1/2} \notag \\
	&= \sqrt{\frac{8\econst \, \log p}{n}} \norm{ \Expect (\mtx{M} \odot \vct{xx}^\adj)^2 }^{1/2}
	+ \frac{8 \econst \log p}{n} \left[ \Expect \max\nolimits_i
	\norm{ \mtx{M} \odot \vct{x}_i \vct{x}_i^\adj }^2 \right]^{1/2}.
\end{align}
The second line follows from the identical distribution of the summands.

\subsubsection{The Matrix Variance and the Maximum Spectral Norm}

All that remains is to calculate the matrix variance that appears in the
first term of~\eqref{eqn:E-ros} and the expected maximum norm that
appears in the second term.
%We postpone these calculations to
%Section~\ref{sec:moments} so we can complete the proof
%of Theorem~\ref{thm:main}.
%The estimate for the matrix variance is the main novelty in our proof.
Lemma~\ref{lem:variance} demonstrates that
\begin{equation} \label{eqn:second-moment-bd}
\Expect (\mtx{M} \odot \vct{xx}^\adj)^2
	\psdle \mu_4^2 \nu^2 \cdot \pnorm{1\to2}{\mtx{M}}^2 \cdot \Id.
\end{equation}
The concentration parameters $\mu_r$ and $\nu$ that characterize $\vct{x}$ are defined in~\eqref{eqn:muq} and~\eqref{eqn:nu}.  Lemma~\ref{lem:schur-norm}
provides a simple deterministic estimate for the remaining Schur product:
\begin{equation} \label{eqn:schur-deterministic}
\norm{ \mtx{M} \odot \vct{xx}^\adj }
	\leq \norm{ \mtx{M} } \infnorm{\vct{x}}^2.
\end{equation}
Introduce the matrix variance bound~\eqref{eqn:second-moment-bd} and the Schur product bound~\eqref{eqn:schur-deterministic} into~\eqref{eqn:E-ros} to obtain
\begin{equation} \label{eqn:E-bd1}
E \leq
	\sqrt{\frac{8\econst \log p}{n}} \cdot \pnorm{1\to 2}{ \mtx{M} } \cdot \mu_4 \nu 
	+ \frac{8 \econst \log p}{n} \cdot \norm{\mtx{M}} \cdot \left[ \Expect \max\nolimits_i \infnorm{\vct{x}_i}^4 \right]^{1/2}.
\end{equation}
To incorporate the semidefinite second moment bound, we have used the fact that the spectral norm is monotone with respect to the order $\psdle$ on the set of positive semidefinite matrices.  This is the first claim in Theorem~\ref{thm:main}.

To establish the remaining claim~\eqref{eqn:thm-expect-max}, we need a bound for the expected maximum of $\infnorm{\vct{x}_i}^4$.  Lemma~\ref{lem:max-value} states that
\begin{equation} \label{eqn:ext-val}
\Expect \max\nolimits_i \infnorm{\vct{x}_i}^4
	\leq \inf_{r \geq 1} \ (np)^{1/r} \cdot \mu_{4r}^4.
\end{equation}
This observation completes the proof.

\section{Computing the Matrix Variance and the Maximum Spectral Norm} \label{sec:moments}

In this section, we complete the calculations that stand at the
center of Theorem~\ref{thm:main}.

%
%  The first two results concern properties
%of the random matrix $\mtx{M} \odot \vct{xx}^\adj$.
%In Section~\ref{subsec:variance}, we develop a semidefinite
%bound on the second moment of this matrix.
%In Section~\ref{sec:schur-norm}, we obtain a simple deterministic
%bound on its norm.  Finally, in Section~\ref{sec:max-value},
%we obtain a bound for the expected maximum of a collection
%of random variables.

\subsection{A Semidefinite Bound for the Matrix Variance} \label{subsec:variance}

First, we study the matrix variance $\Expect( \mtx{M} \odot \vct{xx}^\adj)^2$.
This calculation requires some insight, and it is the main novelty in our proof.
The key idea is that the monotonicity~\eqref{eqn:hdm-monotone} of the Schur product allows us to replace one factor in the product by a scalar matrix.  This act of diagonalization simplifies the estimate tremendously because we erase the off-diagonal entries when we take the Schur product with an identity matrix.

\begin{lemma}[Matrix Variance Bound] \label{lem:variance}
Fix a self-adjoint $p \times p$ matrix $\mtx{M}$.  Let $\vct{x} = (X_1, \dots, X_p)^\adj$
be a random vector.  Then
\begin{equation*}
\Expect (\mtx{M}\odot \vct{xx}^\adj)^2 \psdle \mu_4^2 \nu^2 \pnorm{1\to2}{\mtx{M}}^2 \cdot \Id.
\end{equation*}
The concentration parameters $\mu_4$ and $\nu$ are defined in~\eqref{eqn:muq} and \eqref{eqn:nu}.
\end{lemma}

\begin{proof}
To begin, we perform some algebraic manipulations to consolidate the randomness.  The Schur product identity~\eqref{eqn:Hdm} implies that
\begin{align*}
( \mtx{M}\odot \vct{xx}^\adj)^2
	&= (\diag(\vct{x}) \mtx{M} \diag(\vct{x}))^2 \\
	&= \diag(\vct{x}) (\mtx{M} \diag(\vct{x})^2 \mtx{M}) \diag(\vct{x})
	= (\mtx{M} \diag(\vct{x})^2 \mtx{M}) \odot \vct{xx}^\adj.
\end{align*}
Rewrite the diagonal matrix as a linear combination of matrix units:
$
\diag(\vct{x})^2 = \sum\nolimits_{i} X_i^2 \, \mathbf{E}_{ii}.
$
The bilinearity of the Schur product now yields
$$
( \mtx{M}\odot \vct{xx}^\adj)^2
	= \left[ \mtx{M} \left( \sum\nolimits_i X_i^2 \, \mathbf{E}_{ii} \right) \mtx{M} \right] \odot \vct{xx}^\adj
	= \sum\nolimits_i (\mtx{M} \mathbf{E}_{ii} \mtx{M}) \odot (X_i^2 \, \vct{xx}^\adj).
$$
%The latter calculation depends heavily on the bilinearity of the Schur product.
Take the expectation of this expression to reach
\begin{equation} \label{eqn:variance-2}
\Expect ( \mtx{M}\odot \vct{xx}^\adj)^2
	= \sum\nolimits_i (\mtx{M} \mathbf{E}_{ii} \mtx{M}) \odot [\Expect (X_i^2 \, \vct{xx}^\adj)].
\end{equation}
%Note that we have sequestered all of the randomness in the right-hand factor of each Schur product, so the matrix $\mtx{M}$ does not participate in the expectation.

Next, we invoke the monotonicity~\eqref{eqn:hdm-monotone} of the Schur product to make a diagonal estimate for each summand in~\eqref{eqn:variance-2}:
$$
(\mtx{M} \mathbf{E}_{ii} \mtx{M}) \odot [ \Expect (X_i^2 \vct{x} \vct{x}^*) ]
	\psdle \lambda_{\max}(\Expect (X_i^2 \vct{x} \vct{x}^*)) \cdot (\mtx{M} \mathbf{E}_{ii} \mtx{M}) \odot \Id.
$$
The Rayleigh--Ritz variational formula \cite[Cor.~III.1.2]{Bha97:Matrix-Analysis} allows us to write the maximum eigenvalue as a supremum.  Thus,
\begin{align*}
\lambda_{\max}( \Expect (X_i^2 \vct{x} \vct{x}^*) )
	&= \sup_{\norm{\vct{u}} = 1} \vct{u}^\adj \, [\Expect (X_i^2 \, \vct{xx}^\adj) ] \, \vct{u}
	= \sup_{\norm{\vct{u}} = 1} \Expect \big[  X_i^2 \, \abssq{\vct{u}^\adj \vct{x}} \big] \\
	&\leq \sup_{\norm{\vct{u}} = 1} (\Expect X_i^4)^{1/2} \, (\Expect \abs{\vct{u}^\adj \vct{x}}^4)^{1/2}
%	\leq 2 \kappa(X_i)^2 \sup_{\norm{\vct{u}} = 1} (\Expect \abs{\vct{u}^\adj \vct{x}}^4)^{1/2}
	\leq \mu_4^2 \nu^2.
\end{align*}
The first inequality is Cauchy--Schwarz.  The final inequality follows from the definitions~\eqref{eqn:muq} and~\eqref{eqn:nu} of the concentration parameters.  Combine the last two displays to obtain
\begin{equation} \label{eqn:variance-3}
(\mtx{M} \mathbf{E}_{ii} \mtx{M}) \odot [ \Expect (X_i^2 \vct{x} \vct{x}^*) ]
	\psdle \mu_4^2 \nu^2 \cdot (\mtx{M} \mathbf{E}_{ii} \mtx{M}) \odot \Id.
\end{equation}

To complete our bound for the variance, we introduce~\eqref{eqn:variance-3} into~\eqref{eqn:variance-2}, which delivers
$$
\Expect ( \mtx{M}\odot \vct{xx}^\adj)^2
	\psdle \mu_4^2 \nu^2 \cdot \mtx{M}^2 \odot \Id
	= \mu_4^2 \nu^2 \cdot \diag(\mtx{M}^2)
$$
We can control a positive-semidefinite diagonal matrix using only its maximum entry:
$$
\Expect ( \mtx{M}\odot \vct{xx}^\adj)^2
	\psdle \mu_4^2 \nu^2 \cdot \max\nolimits_i (\mtx{M}^2)_{ii} \cdot \Id
	= \mu_4^2 \nu^2 \pnorm{1\to2}{ \mtx{M} }^2 \cdot \Id
$$
The second relation follows from the fact that the diagonal entries of $\mtx{M}^2$ list the squared $\ell_2$ norms of the columns of $\mtx{M}$, while $\pnorm{1\to2}{\mtx{M}}$ computes the maximum $\ell_2$ norm of a column of $\mtx{M}$.
\end{proof}

\subsection{Norm Bound for a Schur Product} \label{sec:schur-norm}

Next, we present a simple norm bound for the Schur product $\mtx{M} \odot \vct{xx}^\adj$.

\begin{lemma}[Norm Bound for a Schur Product] \label{lem:schur-norm}
Let $\mtx{M}$ be a $p \times p$ self-adjoint matrix, and let $\vct{x}$
be a vector in $\R^p$.  Then 
\begin{equation*}
\norm{ \mtx{M} \odot \vct{xx}^\adj }
	\leq \norm{M} \infnorm{\vct{x}}^2.
\end{equation*}
\end{lemma}

\begin{proof}
The Hadamard product identity~\eqref{eqn:Hdm} yields
$$
\norm{ \mtx{M} \odot \vct{xx}^\adj }
	= \norm{ \diag(\vct{x}) \mtx{M} \diag(\vct{x}) }
	\leq \norm{ \diag(\vct{x})} \norm{ \mtx{M} } \norm{\diag(\vct{x}) }
	= \norm{ \mtx{M} } \infnorm{\vct{x}}^2.
$$
The inequality follows from the submultiplicativity of the spectral norm.
\end{proof}

\subsection{The Expected Maximum Entry among the Sample Vectors} \label{sec:max-value}

Finally, we develop a basic estimate on the expected maximum entry that appears in any of the sample vectors.

\begin{lemma}[The Expected Maximum] \label{lem:max-value}
Consider an i.i.d.~sequence $\{\vct{x}_i\}_{i=1}^n$ of random vectors in $\R^p$.  Then
$$
\left[ \Expect \max\nolimits_i \infnorm{\vct{x}_i}^4 \right]^{1/4}
	\leq \inf_{r \geq 1} \ (np)^{1/4r} \cdot \mu_{4r} .
$$
The diagonal moment parameter $\mu_{4r}$ is defined in~\eqref{eqn:muq}.
\end{lemma}

\begin{proof}
For any $r \geq 1$, Jensen's inequality yields
\begin{equation*}
\left[ \Expect \max\nolimits_i \infnorm{\vct{x}_i}^4 \right]^{1/4}
	\leq \left[ \Expect \max\nolimits_i \infnorm{\vct{x}_i}^{4r} \right]^{1/4r}
	\leq \left[ \sum_{i=1}^n \sum_{j=1}^p \Expect \abs{X_{ij}}^{4r} \right]^{1/4r}
	\leq (np)^{1/4r} \mu_{4r}. 
\end{equation*}
We have written $X_{ij}$ for the $j$th entry of $\vct{x}_i$ and invoked the definition of the diagonal
moment $\mu_{4r}$.
\end{proof}

\appendix

\section{The Matrix Moment Inequality}
\label{app:matrix-concentration}

This appendix contains a proof of Theorem~\ref{thm:new-ros}, the matrix moment inequality that animates our argument.  It costs us no additional energy to prove a result that holds for all moments.

\begin{theorem}[Matrix Moment Inequality] \label{thm:full-concentration}
Assume that $p \geq 3$.

\begin{enumerate}
\item	Suppose that $q \geq 1$, and fix $r \geq \max\{q, 2\log p\}$.  Consider a finite sequence $\{ \mtx{W}_i \}$ of independent, random, positive-semidefinite matrices with dimension $p \times p$.  Then
\begin{equation} \label{eqn:ros-psd}
\left[ \Expect  \norm{ \sum\nolimits_i \mtx{W}_i }^q \right]^{1/q}
	\leq \left[ \norm{ \sum\nolimits_i \Expect \mtx{W}_i }^{1/2}
	+ 2\sqrt{\econst \, r} \left( \Expect \max\nolimits_i \norm{\mtx{W}_i}^q \right)^{1/2q}
	\right]^{2}
\end{equation}

\item	Suppose that $q \geq 2$, and fix $r \geq \max\{q, 2 \log p \}$.  Consider a finite sequence $\{ \mtx{Y}_i \}$ of independent, symmetric, random, self-adjoint matrices with dimension $p \times p$.  Then
\begin{equation} \label{eqn:ros-gen}
\left[ \Expect  \norm{ \sum\nolimits_i \mtx{Y}_i }^q \right]^{1/q}
	\leq \sqrt{\econst \, r} \norm{ \left( \sum\nolimits_i \Expect \mtx{Y}_i^2 \right)^{1/2} }
	+ 2\econst \, r \left( \Expect \max\nolimits_i \norm{\mtx{Y}_i}^{q} \right)^{1/q}.
\end{equation}
\end{enumerate}
\end{theorem}

Theorem~\ref{thm:new-ros} follows from~\eqref{eqn:ros-gen} when we select $q = 2$ and $r = 2\log p$.  We establish Theorem~\ref{thm:full-concentration} below after we provide some comments, preliminary results, and historical background.

Theorem~\ref{thm:full-concentration} shows that the moments of a spectral norm of a sum are controlled by two different quantities.  The first term in~\eqref{eqn:ros-psd} is a matrix mean, while the first term in~\eqref{eqn:ros-gen} is a matrix variance.  These terms reflect the size of moderate deviations, and they depend only weakly on the order $q$ of the moment.  The second term measures the size of the largest summand, and it controls the large deviation behavior of the sum.  These bounds are related to the matrix Rosenthal inequality~\cite{JZ11:Noncommutative-Bennett,MJCFT12:Matrix-Concentration}, and they can be viewed as the moment inequality underlying the matrix Bernstein inequality~\cite[Thm.~1.4]{Tropp11}.

\subsection{Matrix Khintchine Inequality}

The main ingredient in the proof of Theorem~\ref{thm:full-concentration} is the matrix Khintchine inequality.
To our knowledge, this result is the earliest matrix moment inequality.

\begin{proposition}[Matrix Khintchine Inequality] \label{prop:matrix-khintchine}
Suppose that $r \geq 2$.  Consider a finite sequence $\{ \mtx{A}_i \}$ of deterministic, self-adjoint matrices.
Then
$$
\left[ \Expect \pnorm{r}{ \sum\nolimits_i \xi_i \mtx{A}_i }^r \right]^{1/r}
	\leq \sqrt{r} \pnorm{r}{ \left[ \sum\nolimits_i \mtx{A}_i^2 \right]^{1/2} }.
$$
The sequence $\{ \xi_i \}$ consists of independent Rademacher random variables.
\end{proposition}

Lust-Picquard established the first version of the matrix Khintchine inequality
in~\cite{L-P86:Inegalites-Khintchine}, with a weaker
estimate for the constant.  Her subsequent paper with Pisier~\cite{LPP91:Noncommutative-Khintchine}
contains important extensions and refinements.  Buchholz obtained sharp constants for even $r$
in~\cite{Buc05:Optimal-Constants}.
The recent work~\cite[Sec.~7]{MJCFT12:Matrix-Concentration} contains
what may be the easiest proof of Proposition~\ref{prop:matrix-khintchine};
it yields the near-optimal bound $\sqrt{r}$ for the constant.

\subsection{Historical Background on Matrix Concentration Inequalities}

Research on matrix moment inequalities contains several strands that date
back to the late 1990s.  The paper~\cite{PX97:Noncommutative-Martingale}
of Pisier and Xu initiated the field of noncommutative martingale inequalities.
This literature contains many powerful moment bounds that can also be used
to study sums of independent random matrices~\cite{JX03:Noncommutative-Burkholder,
JX05:Best-Constants,JX08:Noncommutative-Burkholder-II,JZ11:Noncommutative-Bennett}.
An early application of this theory appeared in Rudelson's
paper~\cite{Rud99:Random-Vectors}, which uses the matrix Khintchine inequality
to obtain a sample complexity bound for classical covariance estimation.
Many authors in computer science, mathematical
signal processing, and other areas adapted Rudelson's
method~\cite{RV07:Sampling-Large,Tro08:Conditioning-Random,Ver11}.

There is a parallel line of work %from the quantum information theory community
that develops exponential moment inequalities for sums of
random matrices.  This research was initiated in the paper
Ahlswede--Winter~\cite{AW02} and continued in a variety of
other works~\cite{Gro11:Recovering-Low-Rank,Rec11:Simpler-Approach,
Oli10a,Tropp11,Tro11:Freedmans-Inequality,MJCFT12:Matrix-Concentration}.
Over the last few years, these results have started to see wide
application.

As it is stated, Theorem~\ref{thm:full-concentration} seems to be new.
Nevertheless, the result is substantially similar to some previous
matrix concentration inequalities that appear in the literature.
Indeed, we have adapted the argument from Rudelson's paper~\cite{Rud99:Random-Vectors},
the refinements of Rudelson's work in~\cite{RV07:Sampling-Large, Tro08:Conditioning-Random},
and the recent proofs of two matrix Rosenthal
inequalities~\cite{JZ11:Noncommutative-Bennett,MJCFT12:Matrix-Concentration}.

\subsection{Proof of the Matrix Moment Inequality, Part I}

We prove Theorem~\ref{thm:full-concentration} in two steps.  First, we establish 
the inequality~\eqref{eqn:ros-psd} for positive matrices.  In the second stage,
we extend this result to obtain the bound~\eqref{eqn:ros-gen} for general matrices.

To begin, we introduce the quantity of interest:
\begin{equation} \label{eqn:Eq2-step1}
E_q^2 := \left[ \Expect \norm{\sum\nolimits_i \mtx{W}_i}^q \right]^{1/q}
	\leq 2 \left[ \Expect \norm{\sum\nolimits_i \xi_i \mtx{W}_i}^q \right]^{1/q}
	+ \norm{ \sum\nolimits_i \Expect \mtx{W}_i }.
\end{equation}
The inequality follows when we center the sum and apply the standard symmetrization result~\cite[Lem.~6.3]{LT91:Probability-Banach}.
The sequence $\{ \xi_i \}$ consists of independent Rademacher random variables
that are also independent from the sequence $\{\mtx{P}_i\}$.

Let us focus on the first term on the right-hand side of~\eqref{eqn:Eq2-step1}:
$$
F_q := \left[ \Expect \norm{\sum\nolimits_i \xi_i \mtx{W}_i}^q \right]^{1/q}
	\leq \left[ \Expect_{\mtx{W}_i} \left( \Expect_{\xi_i} \pnorm{r}{\sum\nolimits_i \xi_i \mtx{W}_i}^r
	\right)^{q/r} \right]^{1/q}.
$$
The inequality holds because the Schatten $r$-norm dominates the spectral norm,
and we have used the fact $q \leq r$ to apply Jensen's inequality to the inner expectation.
An application of the matrix Khintchine inequality, Proposition~\ref{prop:matrix-khintchine},
delivers the bound
$$
F_q \leq \sqrt{r} \left[ \Expect \pnorm{r}{ \left(\sum\nolimits_i \mtx{W}_i^2\right)^{1/2} }^q \right]^{1/q}.
$$
We proceed through a short chain of inequalities to complete the estimate:
\begin{align*}
F_q &\leq \sqrt{\econst \, r} \left[ \Expect \norm{ \sum\nolimits_i \mtx{W}_i^2 }^{q/2} \right]^{1/q} \\
	&\leq \sqrt{\econst \, r} \left[ \Expect \left( \max\nolimits_i \norm{\mtx{W}_i}^{q/2}
		\cdot \norm{ \sum\nolimits_i \mtx{W}_i }^{q/2} \right) \right]^{1/q} \\
	&\leq \sqrt{\econst \, r} \left[ \Expect \max\nolimits_i \norm{\mtx{W}_i}^q \right]^{1/2q}
		\cdot \left[ \Expect \norm{ \sum\nolimits_i \mtx{W}_i }^q \right]^{1/2q}.
\end{align*}
In the preceding calculation, we first replace the Schatten $r$-norm by the spectral norm,
which results in a loss of at most $p^{1/r} \leq \sqrt{\econst}$.
Since $\mtx{W}_i$ is positive semidefinite,
we can make the bound $\mtx{W}_i^2 \psdle \norm{ \mtx{W}_i } \cdot \mtx{W}_i$ and invoke the monotonicity of the spectral norm on the positive-semidefinite cone to
draw off the maximum norm achieved by any one of the summands.  The last inequality is Cauchy--Schwarz.
We conclude that
\begin{equation} \label{eqn:Fq-bd}
F_q \leq \sqrt{\econst \, r} \left[ \Expect \max\nolimits_i \norm{\mtx{W}_i}^q \right]^{1/2q}
		\cdot E_q
\end{equation}
by identifying a copy of $E_q$.

To complete the argument, introduce~\eqref{eqn:Fq-bd} into the inequality~\eqref{eqn:Eq2-step1}:
$$
E_q^2 \leq 2 \sqrt{\econst \, r} \left[ \Expect \max\nolimits_i \norm{\mtx{W}_i}^q \right]^{1/2q}
	\cdot E_q + \norm{ \sum\nolimits_i \Expect \mtx{W}_i }.
$$
Solutions to the quadratic inequality $x^2 \leq ax + b$ satisfy $x \leq a + \sqrt{b}$.  Therefore,
$$
E_q \leq 2 \sqrt{\econst \, r} \left[ \Expect \max\nolimits_i \norm{\mtx{W}_i}^q \right]^{1/2q}
	+ \norm{ \sum\nolimits_i \Expect \mtx{W}_i }^{1/2}.
$$
This estimate coincides with the bound~\eqref{eqn:ros-psd}.

\subsection{Proof of the Matrix Moment Inequality, Part II}

To establish the second bound~\eqref{eqn:ros-gen}, we apply the matrix Khintchine inequality
to obtain a bound involving positive-semidefinite random matrices, and then we invoke
our first result~\eqref{eqn:ros-psd}.
Indeed, since the summands $\mtx{Y}_i$ are symmetric random variables,
$$
\left[ \Expect  \norm{ \sum\nolimits_i \mtx{Y}_i }^q \right]^{1/q}
	= \left[ \Expect \norm{ \sum\nolimits_i \xi_i \mtx{Y}_i }^q \right]^{1/q}
	\leq \left[ \Expect_{\mtx{Y}_i} \left( \Expect_{\eps_i}
		\pnorm{r}{ \sum\nolimits_i \xi_i \mtx{Y}_i }^r \right)^{q/r} \right]^{1/q}.
$$
The inequality follows from the same considerations as in the proof
of~\eqref{eqn:ros-gen}.  Invoke Proposition~\ref{prop:matrix-khintchine} to
obtain
$$
\left[ \Expect  \norm{ \sum\nolimits_i \mtx{Y}_i }^q \right]^{1/q}
	\leq \sqrt{r} \left[ \Expect \pnorm{r}{ \left( \sum\nolimits_i \mtx{Y}_i^2  \right)^{1/2} }^{q} \right]^{1/q} \\
	\leq \sqrt{\econst \, r} \left[ \Expect \norm{ \sum\nolimits_i \mtx{Y}_i^2 }^{q/2} \right]^{1/q}.
%	\leq \sqrt{\econst \, r} \left[ \Expect \norm{ \sum\nolimits_i \mtx{Y}_i^2 }^{q} \right]^{1/2q}.
$$
In the second step, we have replaced the Schatten $r$-norm with the spectral norm;
the third step follows from Jensen's inequality.
The resulting expression involves a sum of independent, random positive-semidefinite matrices.
Since $q/2 \geq 1$, we can apply~\eqref{eqn:ros-psd} with $\mtx{W}_i = \mtx{Y}_i^2$ to reach
the conclusion~\eqref{eqn:ros-gen}.

\bibliographystyle{IMAIAI}
\bibliography{masked-bib}

\end{document}